\documentclass [10pt,a4paper,reqno]{amsart}

\usepackage[utf8]{inputenc}
\usepackage{amsmath}
\usepackage{lmodern}
\usepackage{amssymb}
\usepackage{xypic}
\usepackage{tikz-cd}
\usepackage{enumerate}

\usepackage{amsmath,amssymb,amsfonts,xfrac,MnSymbol,graphicx,float}
\usepackage{amsthm}
\usepackage{cite}
\usepackage{hyperref}
\usepackage{todonotes}
\usepackage{enumitem}
\usepackage{graphicx}
\usepackage{color}
\usepackage{import}

\numberwithin{equation}{section}

\newtheorem{theorem}{Theorem}[section]

\newtheorem{lemma}[theorem]{Lemma}
\newtheorem{corollary}[theorem]{Corollary}

\newtheorem*{claim*}{Claim}

\theoremstyle{definition}
\newtheorem{definition}[theorem]{Definition}

\newtheorem{remark}[theorem]{Remark}
\newtheorem{example}[theorem]{Example}
\newtheorem{question}[theorem]{Question}

\newtheorem*{notation*}{Notation}

\newcommand{\bfx}{\mathbf{x}}

\newcommand{\calR}{\mathcal{R}}

\newcommand{\calW}{\mathcal{W}}
\newcommand{\calX}{\mathcal{X}}

\DeclareMathOperator{\Zar}{Zar}
\DeclareMathOperator{\Zars}{Zar^{+}}
\DeclareMathOperator{\FM}{FM}

\begin{document}

\title{A note on the Zariski topology on groups} 

\author{Gil Goffer}
\address{Department of Mathematics, University of California, San Diego, La Jolla, CA 92093, USA}
\email{ggoffer@ucsd.edu}

\author{Be'eri Greenfeld}
\address{Department of Mathematics, University of Washington, Seattle, WA 98195-4350, USA}
\email{grnfld@uw.edu}


\maketitle

\begin{abstract} 
We show that the semigroup Zariski topology on a group can be strictly coarser than the group Zariski topology on it, answering a question from \cite{E-P}.
\end{abstract}

\section{Introduction} \label{sec:Intro}


By analogy with algebraic geometry, one defines Zariski topologies on groups and on semigroups\footnote{Not to be confused with the geometric Zariski topology on algebraic groups.}: closed sets are given by common solutions to word maps, which play the role of polynomial functions \cite{Bryant,DS,DT,E-P,Markov_Eng}.

Let $S$ be a semigroup and $\left<\mathbf{x}\right>^+=\{1,\mathbf{x},\mathbf{x}^2,\dots\}$ a free monogenic monoid. The free product $S[\mathbf{x}]:=S*\left<\mathbf{x}\right>^+$ plays the role of a `polynomial ring', and any $W(\bfx)\in S[\mathbf{x}]$ 
takes the form:
\begin{equation*}\label{eq:w}
  W(\bfx) = s_0\mathbf{x}^{i_1}s_1\cdots s_n \mathbf{x}^{i_n} s_{n+1}  
\end{equation*}
for some integers $i_1,\dots,i_n\geq 0$ and some $s_0,\dots,s_{n+1}\in S$. The elements $s_0,\dots,s_{n+1}$ are called the \textit{coefficients} of $W(\bfx)$. We also interpret $W(\bfx)$ as a function $W:S\rightarrow S$ by evaluating $\bfx$. Namely, for $x\in S$ and $W(\bfx)$ as above, $W(x)=s_0x^{i_1}s_1\cdots s_n x^{i_n} s_{n+1}$. 

For each $W(\bfx),V(\bfx)\in S[\mathbf{x}]$, let:
\[
\mathcal{O}_{W,V} := \{x\in S\ |\ W(x) \neq V(x) \text{ in $S$} \}.
\]
The \textit{semigroup Zariski topology} (or positive Zariski topology) on $S$, denoted $\Zars(S)$, is defined by taking: 
\[\{ \mathcal{O}_{W,V} |\ W(\bfx),V(\bfx)\in S[\mathbf{x}]\}\] 
as a subbase of open sets. In \cite{E-P}, the authors characterized $\Zars$ for some important classes of semigroups.

Now let $\Gamma$ be a group and $\left< \mathbf{x}\right>=\{\dots, \bfx^{-1},1,\bfx,\dots\}$ an infinite cyclic group. In the same vein as above, any $W(\bfx)$ in the free product $\Gamma[\mathbf{x}]:=\Gamma*\langle \mathbf{x}\rangle$ 
takes the form:
\[
W(\bfx) = g_0\mathbf{x}^{i_1}g_1\cdots g_n \mathbf{x}^{i_n} g_{n+1},
\]
where now $i_1,\dots,i_n$ are arbitrary integers and the coefficients $g_0,\dots,g_{n+1}$ are elements of $\Gamma$. Also here, $W(\bfx)$ is interpreted as a function $W:\Gamma\rightarrow \Gamma$ by evaluating $\mathbf{x}$, so for $x\in \Gamma$, $W(x) = g_0x^{i_1}g_1\cdots g_n x^{i_n} g_{n+1}$. For any $W(\bfx)\in \Gamma[\mathbf{x}]$, let:
\[
\mathcal{O}_{W} := \{x\in \Gamma\ |\ W(x) \neq 1 \text{ in $\Gamma$}\}.
\]
The \textit{group Zariski topology} (or simply, \textit{Zariski topology}) on $\Gamma$, denoted $\Zar(\Gamma)$, is defined by taking:
\[\{ \mathcal{O}_{W} |\ W(\bfx)\in \Gamma[\mathbf{x}]\}\] 
as a subbase of open sets. The group and semigroup Zariski topologies are essential in the logical-geometric analysis of groups and semigroups. For instance, see \cite{Ra,KhM,Sela,Sela_sgp}. 

It is natural to compare the Zariski topology on a given group with the semigroup Zariski topology on it.
Evidently,
\[
\Zars(\Gamma) \subseteq \Zar(\Gamma).
\]
The following question naturally arises:
\begin{question}[{\cite[Question~2.8]{E-P}}]
Is $\Zars(\Gamma)=\Zar(\Gamma)$ for any group $\Gamma$?
\end{question}

Equality holds for abelian groups, for groups of finite exponent and for free groups (see \cite{CR} for a concrete description of $\Zar(F)$ for any free group $F$).
In this note we give a self-contained construction that demonstrates a negative answer to this question.

\begin{theorem} \label{thm:main1}
There exists a countable group $\Gamma$ for which $\Zars(\Gamma) \neq \Zar(\Gamma)$.
\end{theorem}

Section \ref{sec:preliminaries} includes necessary preliminaries on presentations of groups and the small cancellation property, and in Section \ref{sec:Zar neq Zar+} we prove Theorem \ref{thm:main1}. 
In the end of this note we address the relation between the Zariski and Fr\'echet-Markov topologies on semigroups and also include, as a side remark, an answer to \cite[Question~2.7]{E-P}.

\subsection*{Acknowledgements.} 
We are gratefully thankful to Chloé Perin for her advice regarding equations over free groups, and for bringing the interesting paper \cite{CR} to our attention.





\section{Presentations of groups}\label{sec:preliminaries}
Let $B$ be a set, to which we refer to as an alphabet. We denote by $F(B)$ the free group generated by elements of $B$. By convention, all words in $F(B)$ are freely reduced, and therefore when we write $W=U$, it means that the words $W,U\in F(B)$ are equal letter by letter. For a word $W\in F(B)$, we denote by $|W|$ its length.
We write $U\subseteq W$ to say that $U$ is a \emph{subword} of $W$, namely that there exists $V_1,V_2\in F(B)$ such that $W$ equals to the concatenation $V_1UV_2$. If $V_1=e$ then $U$ is called an \emph{initial segment} of $W$. We write $U\subset W$ if $U\subseteq W$ but $U\neq W$. A word $U\in F(B)$ is said to \emph{appear} or to \emph{occur} in a word $W\in F(B)$ if $U\subseteq W$. A letter $b\in B$ is said to appear or to occur in $W$ if $U=\{b\}\subseteq W$.

Let $b\in B$. A word $W\in F(B)$ is said to be \emph{positive in $b$} if $W$ does not contain any appearance of the letter $b^{-1}$, and \emph{negative in $b$} if it does not contain any appearance of the letter $b=b^{+1}$. 
Of course, any word not containing the letter $b$ is both positive in $b$ and negative in $b$, and many words, for example the word $bab^{-1}$, are neither positive, nor negative, in $b$.

Let $W=b_1b_2\cdots b_n$ be a word in $F(B)$ and let $j\in \{0,1,\dots,n-1\}$.
The \emph{$j^{th}$ cyclic permutation} of $W$ is the word $b_{j+1}b_{j+2}\cdots b_n b_1\cdots b_j$. 
A set $\calR\subset F(B)$ is called \emph{symmetrized} if for every $W\in \calR$, all cyclic permutations of $W$ and of $W^{-1}$ are included in $\calR$.
A nontrivial word $U\in F(B)$ is called a \emph{piece} with respect to a set $\calR\subset F(B)$ if there exist two distinct elements $W_1,W_2\in \calR$ that have $U$ as maximal common initial segment. 

The following condition is often called the \textit{metric small cancellation condition}. 
\begin{definition}
Let $0 < \lambda < 1$. 
Let $\calR\subset F(B)$ be a symmetrized set of reduced words. The presentation $$\Gamma = \langle B \mid \calR \rangle$$ is said to satisfy \emph{$C'(\lambda)$ small cancellation condition} if whenever $U$ is a piece with respect to $\calR$ and $U$ is a subword of some $W\in \calR$, then: $$|U|<\lambda|W|.$$
\end{definition}

The main result regarding the metric small cancellation condition is the following statement (see \cite[Theorem 4.4 in Ch. V]{LC}).

\begin{lemma}[Greendlinger's lemma]\label{lem:Greendlinger's}
Let $\calR\subset F(B)$ be a symmetrized set. 
Let $G=\langle B \mid \calR\rangle $ be a group presentation satisfying the $C'(\lambda)$ small cancellation condition, where $0 \leq \lambda \leq \frac{1}{6}$. Let $V\in F(B)$ be a non-trivial freely reduced word such that $V = 1$ in $G$. Then there is a subword $U$ of $V$ and a word $R\in \calR$ such that $U$ is an initial segment of $R$ and such that
\[ |U| > (1-3\lambda)|R|. \]
\end{lemma}

\section{Proof of main theorem}
\label{sec:Zar neq Zar+}
Fix an 
even integer $k\geq 8$. Let $A=\{a_1,a_2\dots,a_{k+1}\}$ and $X=\{x_1,x_2,\dots\}$. 
Let $F(A\cup X)= \left<a_1,\dots,a_{k+1},x_1,x_2,\dots\right>$ be the free group generated by $A\cup X$. For $i\in \mathbb{N}$ define: \[w_i:=a_1x_i^{-1}a_2x_ia_3x_i^{-1}a_4x_ia_5\cdots a_{k-1}x_i^{-1}a_kx_i.\]

For $j\in \{0,\dots,2k-1\}$ and $i\in \mathbb{N}$, denote by $w_{i,j}$ the $j^{th}$ cyclic permutation of $w_i$. Then the set $\calW=\{w_{i,j}^{\pm 1}\}_{i,j}$ is symmetrized.
Consider the group $\Gamma$ defined as:
\[
\Gamma = 
\langle A\cup X \mid w_1,w_2,\dots\rangle =  \langle A\cup X \mid \calW \rangle
\]



\begin{lemma} \label{lem:Gamma_r is SC}
$\Gamma=\langle A\cup X \mid \calW \rangle$ satisfies $C'(\frac{1}{2k})$-small cancellation condition.
\end{lemma}
\begin{proof}
Any piece $U$ with respect to $\calW$ that occurs as a common initial segment of $W_1=w_{i,j}^{\pm 1},W_2=w_{k,l}^{\pm 1}$ with $i\neq k$ is of the form $U=a_r^{\pm 1}$ for some $1\leq r\leq k$. If $i=k$, then a common initial segment of $W_1$ and $W_2$ must be of the form $U=x_i^{\pm 1}$. 


Therefore, any piece $U$ with respect to $\calW$ has length $1$. Since all words in $\calW$ have length $2k$, the claim is proved.
\end{proof}


In what follows, we often talk about words $U,W\in F(A\cup X)$ that are equal in the group $\Gamma$. For clarification, we note that when we write \emph{$W=U$ in $\Gamma$}, we mean that $W$ and $U$ represent the same element in $\Gamma$.
This does not imply that $W$ and $U$ are equal letter by letter. We also clarify that when we say that a word $W=b_1b_2\cdots b_n\in F(B)$ decomposes as $W=W_1W_2$ this means that for some $i\in \{1,\dots,n-1\}$, $W_1=b_1\cdots b_i$ and $W_2=b_{i+1}\cdots b_n$.

\begin{lemma}\label{lem:V=1 dont part to V+V-}
Let $m\in \mathbb{N}$ and $V\in F(A\cup X)$, and suppose that $x_m$ appears in $V$. Suppose further that there exist words $V^+,V^-\in F(A\cup X)$ such that $V^+$ is positive in $x_m$, $V^-$ is negative in $x_m$, and $V$ decomposes as $V=V^+V^-$. Then $V\neq 1$ in $\Gamma$.
\end{lemma}
\begin{proof}
Suppose toward contradiction that $V= 1$ in $\Gamma$. By an iterative application of Lemma \ref{lem:Greendlinger's}, we inductively construct a sequence $V=V_0,V_1,V_2,\dots,V_n$ of words such that:
\begin{enumerate}
    \item $|V_0|>|V_1|>\dots>|V_n|$
    \item $V_0=V_1=\dots=V_n=1$ in $\Gamma$
    \item $V_n=1$
    \item Each word $V_{t+1}$ is achieved from its preceding $V_t$ 'by one relation only'. Namely, for each $t\in \{0,\dots,n-1\}$ there exist $\epsilon_t\in \{\pm 1\}$, $i_t\in \mathbb{N}$, and $j_t\in \{0,1,\dots,2k-1\}$ such that $W_t=w^{\epsilon_t}_{i_t,j_t}$ decomposes as $W_t=U_tS_t$, and the word $V_{t+1}$ is achieved from $V_t$ by replacing an occurance of the subword $U_t\subseteq V_t$ with $S_t^{-1}$.   
\end{enumerate}

As a base, set $V_0=V$. Suppose that $V_t$ is already defined.
By Lemma \ref{lem:Gamma_r is SC}, $\calW$ satisfies $C'(\frac{1}{2k})$-small cancellation condition. Therefore, by Lemma \ref{lem:Greendlinger's}, there exist a word $W_t:=w^{\varepsilon_t}_{i_t,j_t}\in \calW$ for some $\varepsilon_t\in \{\pm 1\}$, $i_t\in \mathbb{N}$, and $j_t\in\{0,\dots,2k-1\}$, such that the word $W_t$ decomposes as $W_t=U_tS_t$, 
where $U_t$ is a subword of $V_t$, and
\[|U_t|>\left(1-\frac{3}{2k}\right)|W_t|.\]
Note that by the choice of $k$, $|U_t|>\frac{1}{2}|W_t|$, and therefore $|U_t|>|S_t|=|S_t^{-1}|$. 
Set $V_{t+1}$ to be the word obtained from $V_t$ by replacing the occurrence of the subword $U_t$ by $S_t^{-1}$. (If the subword $U_t$ occurs in $V_t$ more than once, we replace one of the occurrences.). Thus $|V_{t+1}|<|V_t|$. Moreover, since $U_t=S_t^{-1}$ in $\Gamma$, we have that $V_{t+1}=V_t$ in $\Gamma$. Since each word in the sequence is strictly shorter than its predecessor, there exists $n$ with $|V_n|=0$, namely, with $V_n=1$ (as words, and not just in $\Gamma$). This completes the inductive construction.



Recall that by assumption, $x_m$ occurs in $V_0=V$. It follows that for some $t\in \{0,1,\dots,n-1\}$ we have $i_t=m$. 
Indeed, supposing the converse, namely that $i_t\neq m$ for all $t$, then for all $t$ the subword $U_t\subset W_t$ does not contain the letter $x_m$. 
Then $x_m$ must occur in each of the words $V_1,V_2,\dots,V_n$, contradicting the fact $V_n=1$.

Let $t$ be the minimal number for which $i_t=m$. 
So $W_t$ is a cyclic permutation of $w_m$ or of $w_m^{-1}$, and $V_t$ contain the initial segment $U_t$ of $W_t$. By the choice of $k$, the subword $U_t$ of $W_t$ is of length at least $(1-\frac{3}{2k})|W_t|=2k-3\geq 13$. Based on the structure of $w_m^{\pm 1}$, $U_t$ must contain either the subword:
$$U=x_ma_{l}x_m^{-1}a_{l+1}x_m$$
or its inverse $U^{-1}$, for some $l\in \{1,\dots,k-1\}$. Without loss of generality suppose that $U_t$ contains that subword $U$. In particular, $U$ is a subword of $V_t$.

By assumption, $V=V^+V^-$.
By minimality of $t$, the letter $x_m$ does not appear in any $U_s$ for $s<t$. It follows that also $V_t$ can be decomposes as $V_t=V_t^+V_t^-$.
Decompose $U$ as $U=U^+U^-$ such that $U^+\subset V_t^+,U^-\subset V_t^-$ (each of $U^+,U^-$ might be empty). Then $U^+$ is positive in $x_m$ and $U^-$ is negative in $x_m$.
But this is a contradiction to the structure of $U$, as it is easily seen that for any decomposition $U=U^+U^-$, either $U^+$ or $U^-$ involves both positive and negative powers of $x_m$.
\end{proof}

\begin{lemma} \label{lem:dense}
The set $X$ is dense in $\Gamma$ with respect to $\Zars(\Gamma)$.
\end{lemma}

\begin{proof}
It suffices to prove that every basic closed set in $\Zars(\Gamma)$ that contains $X$ is equal to $\Gamma$ itself. 

First, let $C\subset \Gamma$ be a non-empty sub-basic closed set in $\Zars(\Gamma)$. So $C$ is given by:
\[
C = \{\gamma\in \Gamma\ |\ P(\gamma)=Q(\gamma)\}
\]
where $P(\bfx),Q(\bfx)\in \Gamma[\bfx]$. 
Since $C\notin \{\emptyset,\Gamma\}$, at least one of $P(\bfx),Q(\bfx)$ is not constant.
There are only finitely many words in $F(X\cup A)$ that appear as coefficients in $P(\bfx)$ or in $Q(\bfx)$. Therefore, there exists large enough $m\in \mathbb{N}$ such that $x_m$ does not appear in any of the coefficients of $P(\bfx),Q(\bfx$).

We show that $x_m\notin C$. Indeed, suppose the converse. Then $P(x_m)= Q(x_m)$ in $\Gamma$, and so $V:=P(x_m)(Q(x_m))^{-1}=1$ in $\Gamma$. But note that $V^+:=P(x_m)$ is positive in $x_m$ and $V^-:=(Q(x_m))^{-1}$ is negative in $x_m$. Moreover, $V$ contains $x_m$, since at least one of $P(\bfx),Q(\bfx)$ is non-constant. This contradicts Lemma \ref{lem:V=1 dont part to V+V-}.

This proves that any sub-basic non-empty proper subset $C\subset \Gamma$ has a finite intersection with $X$.

Finally, since any basic closed set is a finite union of sub-basic closed sets, it follows that no proper closed subset contains $X$.
\end{proof}

We are now ready to prove the main theorem.

\begin{proof}[{Proof of Theorem \ref{thm:main1}}]
The set:
\[
C := \{ x\in \Gamma\ |\ a_1x^{-1}a_2xa_3x^{-1}a_4xa_5\cdots a_{k-1}x^{-1}a_k x=1 \}
\]
is closed with respect to $\Zar(\Gamma)$ by definition, and $X\subseteq C$. However, $a_{k+1}\notin C$, since $\left<a_{k+1}\right>$ is a free factor of $\Gamma$. Therefore, $X$ is not dense with respect to $\Zar(\Gamma)$. By Lemma \ref{lem:dense}, $X$ is dense with respect to $\Zars(\Gamma)$, proving that $\Zar(\Gamma)\neq \Zars(\Gamma)$.
\end{proof}


\begin{remark}
Non-empty Zariski open sets are thought of as `big' or generic subsets; however, they can be very sparse with respect to the metric geometry of a given group. Indeed, in \cite{GG} we constructed a finitely generated group $\Gamma$ with a free subgroup $F$, such that the union over all conjugates of $F$ is a Zariski open set, yet has measure zero with respect to the limit of counting measures on balls in the Cayley graph of $\Gamma$ or with respect to any non-degenerate random walk on it.
\end{remark}

We conclude with an observation about other topologies on semigroups and groups. 
A topology on a semigroup $S$ is called a \textit{semigroup topology} if the multiplication map $\mu\colon S\times S\rightarrow S$ is continuous, where $S\times S$ is endowed with the product topology. Similarly, one defines a group topology, further requiring that the inversion map is continuous. See \cite{KOO,BL} for interesting examples of groups whose group topologies are extremely restricted.
While the Zariski topology need not be a semigroup topology, it is contained in the intersection of all $T_1$ semigroup topologies, called the Fr\'echet-Markov topology; a similar statement holds for groups too.
In \cite{E-P}, the authors ask:
\begin{question}[{\cite[Question 2.7]{E-P}}]
    Is the Fr\'echet-Markov topology always contained in the Zariski topology?
\end{question}

For groups, this question was posed by Markov, and settled in the affirmative by him in the countable case \cite{Markov,Markov_Eng}, but in the negative in the uncountable case \cite{Hesse} (attributed e.g. in \cite{Si}). We observe that for semigroups, the answer is negative even in the countable case.

\begin{example} 
Consider the following monomial semigroup: 
\[
S=\left<x_1,x_2,\dots,y_1,y_2,\dots\ |\ x_iy_i=0\ \forall i=1,2,\dots\right>
\]
We claim that the Fr\'echet-Markov topology on $S$ is not contained in the Zariski topology on it.

We first show that any proper non-empty Zariski closed subset $C\subset S$ has a finite intersection with $\{x_1,x_2,\dots,y_1,y_2,\dots\}$. Indeed, it suffices to prove it for a sub-basic closed set in $\Zars(S)$. Let $C=\{s\in S\ |\ P(s)=Q(s)\}$ be a non-empty, proper, sub-basic closed set in $\Zars(S)$, where $P(\bfx),Q(\bfx)\in S[\bfx]$. 
Pick $m\in \mathbb{N}$ large enough such that the letters $x_i,y_i$ with $i\geq m$ do not appear in any of the coefficients of $P(\bfx),Q(\bfx)$. Fix some $i\geq m$ and suppose that $x_i\in C$. Consider the quotient semigroup $\bar{S}=S/ \llangle y_i\rrangle$ (here $\llangle y_i \rrangle$ stands for the congruence generated by equating $y_i$ to zero). We can view $P(\bfx),Q(\bfx)$ as elements of $\bar{S}[\bfx]$, and $x_i\in C$ implies that $P(x_i)=Q(x_i)$ also in $\bar{S}$. But $\langle x_i \rangle\leq \bar{S}$ is a free factor and so $P(x_i)=Q(x_i)$ implies $P(\bfx)=Q(\bfx)$. Thus $C=S$ is not a proper subset, contradicting the assumption. Hence $x_i\notin C$, and a similar argument shows that $y_i\notin C$. It follows that the intersection of $C$ with $\{x_1,x_2,\dots,y_1,y_2,\dots\}$ is finite, as claimed.

Let $D=\{(a,b)\in S\times S\ |\ ab=0\}=\mu^{-1}\left(\{0\}\right)\subset S\times S$ where $\mu\colon S\times S\rightarrow S$ is the multiplication map. Clearly $D\subset S\times S$ is closed in the product topology corresponding to any $T_1$ semigroup topology on $S$, and is therefore closed with respect to the product topology where $S$ is equipped with the Fr\'echet-Markov topology. However, $D$ is not closed in the product topology where $S$ is equipped with the Zariski topology. Assume otherwise. Since $(x_1,y_2)\notin D$, it follows that there exist open subsets $U,V\subset S$ such that $x_1\in U,y_2\in V$ and $D\cap (U\times V)=\emptyset$. By the above argument, for some sufficiently large $i$ it holds that $x_i\in U$ and $y_i\in V$, and so $(x_i,y_i)\in D \cap (U\times V)$, a contradiction.
\end{example}

\end{document}